\documentclass[ECP]{ejpecp} % replace ECP by EJP if needed.  

\usepackage{subcaption}

\SHORTTITLE{Connective constant for a weighted self-avoiding walk on~$\mathbb{Z}^2$} 

\TITLE{Connective constant for a weighted self-avoiding walk on~$\mathbb{Z}^2$} % \thanks is optional. Insert line breaks with \\

\AUTHORS{%
  Alexander~Glazman
  \footnote{University of Geneva, Switzerland ; 
  Chebyshev Laboratory, St.\,Petersburg, Russian Federation
   \BEMAIL{alexander.glazman@unige.ch}} }
 
\KEYWORDS{weighted self-avoiding walks ; connective constant ; integrable weights ; Yang-Baxter equation ; parafermionic observable} % Separate items with ;

\AMSSUBJ{82B41} % Edit. Separate items with ;
\AMSSUBJSECONDARY{60J67 ; 60K35 ; 82D60} % Optional, separate items with ;

\SUBMITTED{October 5, 2014} % Edit.
\ACCEPTED{September 15, 2015} % Edit.

\ARXIVID{1402.5376} % Edit.
\VOLUME{20}
\YEAR{2015}
\PAPERNUM{86}
\DOI{v20-3844}

\ABSTRACT{We consider a self-avoiding walk on the dual~$\mathbb{Z}^2$ lattice.
This walk can traverse the same square twice but cannot cross the same edge more than once.
The weight of each square visited by the walk depends on the way the walk passes 
through it and the weight of the whole walk is calculated as a product of these weights. 
We consider a family of critical weights parametrized by angle~$\theta\in[\frac{\pi}{3},\frac{2\pi}{3}]$. 
For~$\theta=\frac{\pi}{3}$, this can be mapped to the self-avoiding walk on the honeycomb lattice. 
The connective constant in this case was proved to be equal to~$\sqrt{2+\sqrt{2}}$ 
by Duminil-Copin and Smirnov in~\cite{DS10}. We generalize their result.}

\newcommand{\loops}{\mathrm{loops}}

\begin{document}

%%%%%%%%%%%%%%%%%%%%%%%%%%%%%%%%%%%%%%%%%%%%%%%%%%%%%%%%%%%%%%%%%%%
%%                                                               %%
%% No need for \maketitle.                                       %%
%%                                                               %%
%%%%%%%%%%%%%%%%%%%%%%%%%%%%%%%%%%%%%%%%%%%%%%%%%%%%%%%%%%%%%%%%%%%

%%%%%%%%%%%%%%%%%%%%%%%%%%%%%%%%%%%%%%%%%%%%%%%%%%%%%%%%%%%%%%%%%%%
%%                                                               %%
%% Please replace what follows by the body of your article       %%
%% (up to the bibliography):                                     %%
%%                                                               %%
%%%%%%%%%%%%%%%%%%%%%%%%%%%%%%%%%%%%%%%%%%%%%%%%%%%%%%%%%%%%%%%%%%%

\section{Introduction}
Self-avoiding walks (i.e. visiting each vertex at most once) 
were proposed by P.~Flory and W.\,J.\,C.~Orr \cite{F,O}. 
These walks turned out to be a very interesting object leading to rich mathematical theories, 
see~\cite{MS,BDGS}, and raising important challenges (it is difficult to understand because of its non-markovity). 
There are not  so many rigorous statements in this field, one of 
the main conjectures being convergence to~$\text{SLE} (8/3)$. 
Some progress in this direction was achieved by G.~Lawler, O.~Schramm and W.~Werner, 
who proved in~\cite{LSW} that if the scaling-limit of self-avoiding walk exists and is conformally invariant, 
then it is~$\text{SLE}(8/3)$. In~1984, B.~Nienhuis nonrigorously derived in~\cite{N82} 
that the connective constant for the hexagonal lattice equals to~$\sqrt{2+\sqrt{2}}$. 
This has been proved recently by H.~Duminil-Copin and S.~Smirnov in~\cite{DS10}. 
Since the self-avoiding walk on the square lattice does not seem to be integrable, 
it is not reasonable to expect any explicit formula for the connective constant in this case. 
Nevertheless, one can study natural variations of the model, for instance by introducing additional weights.

We fix~$\theta\in\left[\frac{\pi}{3},\frac{2\pi}{3}\right]$ and consider the self-avoiding walk 
on~$\Lambda$~--- the skewed~$\mathbb{Z}^2$ lattice with edges having length~$1$ 
and all plaquets having angles~$\theta$ and~$\pi-\theta$.

To be precise this will be a curve starting and ending at the midpoints of edges, 
intersecting edges at right angles and having in each plaquet either one straight 
line connecting two opposite edges or two arcs surrounding opposite vertices or 
one arc or just no arcs (see fig.~\ref{figWeights}). Each rhombus has a weight 
according to the configuration of arcs inside it (see fig.~\ref{figWeights}):
\begin{itemize}
\item empty plaquet has weight~$1$,
\item plaquet with an arc of angle~$\theta$ has weight~$u_1$,
\item plaquet with an arc of angle~$\pi-\theta$ has weight~$u_2$,
\item plaquet with a straight line has weight~$v$,
\item plaquet with two arcs of angle~$\theta$ has weight~$w_1$,
\item plaquet with two arcs of angle~$\pi-\theta$ has weight~$w_2$.
\end{itemize}

\begin{figure}
\centering
\begin{subfigure}{.7\textwidth}
  \centering
  \includegraphics[scale=0.85]{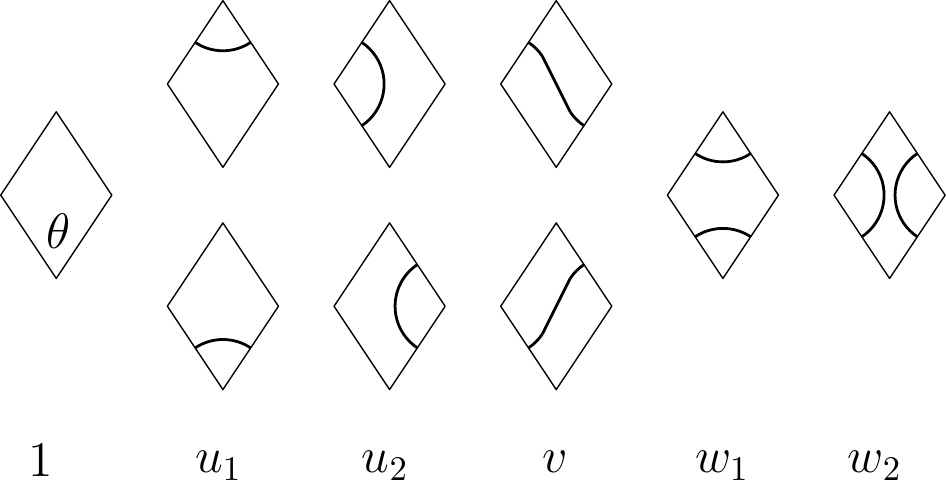}
\end{subfigure}%
\begin{subfigure}{.3\textwidth}
  \centering
  \includegraphics[scale=0.85,page=1]{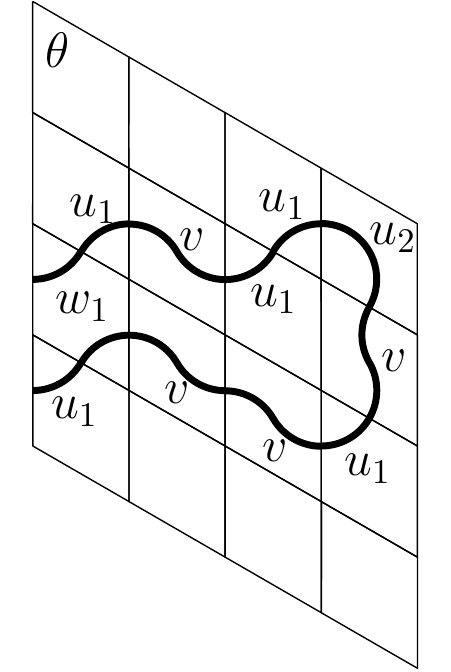}
\end{subfigure}
\caption{Different ways of passing a rhombus with their weights and an
    example of a walk of weight~$u_1(\theta)^5u_2(\theta)v(\theta)^4w_1(\theta)$
    and length~12.}
\label{figWeights}
\end{figure}

The weight of the whole walk is calculated as the product of weights of the plaquets. 
Denote one of the mid-edges of the lattice by~$0$.
The partition function is equal to the sum of the weights of all self-avoiding walks on~$\Lambda$ starting at~$0$:
\begin{align*}
\omega(\gamma)&=\prod_{r\,-\,rhombus}{\omega (r)}\, ,\\
Z(u_1,u_2,v,w_1,w_2)&=\sum_{\gamma}{{\omega(\gamma)}}\, .
\end{align*}
\noindent Let us consider
$$
\tilde{c}_n=\frac{1}{u_1^n}\sum_{|\gamma|=n}{\omega(\gamma)},
$$
where by~$|\gamma|$ we mean the number of arcs in~$\gamma$ (straight passing
of a rhombus counted as one arc).
\noindent 
By definition, we find
$$
Z(u_1,u_2,v,w_1,w_2)=\sum_{n=0}^{\infty}{\tilde{c}_n u_1^n}.
$$

We are now in a position to state our main result:

\begin{theorem}\label{th_CC}
There exists a family of weights ($u_1$, $u_2$, $v$, $w_1$, $w_2)_\theta$ 
parametrized by~$\theta\in [\frac{\pi}{3},\frac{2\pi}{3}]$ such that for these weights
$\lim_{n\to\infty}\sqrt[n]{\tilde{c}_n}$ exists and is equal to~$\frac{1}{u_1}$.

Furthermore, these weights can be calculated explicitly:

\begin{align}
u_1&=\frac{\sin(\frac{5\pi}{4})\sin(\frac{5\pi}{8}+\frac{3\theta}{8})}{\sin(\frac{5\pi}{4}+\frac{3\theta}{8})\sin(\frac{5\pi}{8}-\frac{3\theta}{8})}\, ,
\label{eq_x1c}\\
u_2&=\frac{\sin(\frac{5\pi}{4})\sin(\frac{3\theta}{8})}{\sin(\frac{5\pi}{4}+\frac{3\theta}{8})\sin(\frac{5\pi}{8}-\frac{3\theta}{8})}\, ,
\label{eq_x2c}\\
v&=\frac{\sin(\frac{5\pi}{8}+\frac{3\theta}{8})\sin(-\frac{3\theta}{8})}{\sin(\frac{5\pi}{4}+\frac{3\theta}{8})\sin(\frac{5\pi}{8}-\frac{3\theta}{8})}\, ,
\label{eq_vc}\\
w_1&=\frac{\sin(\frac{5\pi}{8}+\frac{3\theta}{8})\sin(\frac{5\pi}{4}-\frac{3\theta}{8})}{\sin(\frac{5\pi}{4}+\frac{3\theta}{8})\sin(\frac{5\pi}{8}-\frac{3\theta}{8})}\, ,
\label{eq_w1c}\\
w_2&=\frac{\sin(\frac{15\pi}{8}+\frac{3\theta}{8})\sin(-\frac{3\theta}{8})}{\sin(\frac{5\pi}{4}+\frac{3\theta}{8})\sin(\frac{5\pi}{8}-\frac{3\theta}{8})}\, .
\label{eq_w2c}
\end{align}
\end{theorem}

\begin{theorem}
\label{th_different_length}
Consider another way to define~$|\gamma|$: a~$\theta$-arc has length~1,
a~$(\pi-\theta)-arc$ and a straight segment have any positive integer length
(possibly different from each other).
Then Theorem~\ref{th_CC} remains true,
i.\,e. the limit of~$\sqrt[n]{\tilde{c}_n}$ is equal to~$\frac{1}{u_1}$.
\end{theorem}

\begin{remark}
The case~$\theta = \frac{\pi}{3}$ corresponds to the honeycomb lattice
and there is a way to define~$|\gamma|$ in such a way
that Theorem~\ref{th_different_length} computes the connective constant
of the honeycomb lattice (see Section~\ref{section_honeycomb}).
\end{remark}

The weights~\eqref{eq_x1c}-\eqref{eq_w2c} were discovered 
by B.~Nienhuis~\cite{N90} in~$1990$ as solutions of the Yang-Baxter equation. 
They were rediscovered by J.~Cardy and Y.~Ikhlef~\cite{CI} in~$2009$  
as the weights for which the parafermionic observable satisfies some particular equations.
For the connection between these two approaches, see~\cite{AB,IWWZ}.
See also~\cite{dGLR} for the weights on the boundary.
We should just mention here that in~\cite{N90} and~\cite{CI} 
a more general case is considered~--- the $O(n)$ model with~$n\in[-2,2]$ 
(the self-avoiding walk is a particular case of this model for~$n=0$). 
Unfortunately, the weights written there contain some minor misprints, 
so for completeness we include a correct version of the weights in Section~\ref{section_O(n)}.

In the case~$\theta = \frac{\pi}{2}$ the weights are symmetric, 
i.\,e. $u_1 = u_2$ and~$w_1 = w_2$.
One can view a walk as a self-avoiding walk on~$\mathbb{Z}^2$
which is allowed to touch itself but each time gets penalised by~$w_1/u_1^2 \approx 0.675$
and that gets penalised by~$v/u_1 \approx 0.785$ for each vertex it passes without a turn.
Theorem~\ref{th_CC}  confirms the conjecture~\cite{B} that the asymptotics of~$\sqrt[n]{\tilde{c}_n}$ 
is equal to
$$
\frac{1}{u_1(\pi/2)} = \sqrt{3+\frac{1}{2}\sqrt{26+7\sqrt{2}}} = 2.448\dots
$$
This is below the predicted~\cite{GE88} value $\approx 2.638$ for a connective constant of~$\mathbb{Z}^2$.

One can consider~$\theta < \tfrac{\pi}{3}$ (or~$\theta > \tfrac{2\pi}{3}$)
but the weight~$w_2$ (or~$w_1$) becomes negative,
so we do not address this question here.

Another interesting question is the value of the critical fugacities for
walks in a half-plane interacting with the boundary.
For the self-avoiding walk in the half-plane insertion of a fugacity means favouring each 
additional visit of the border.
One can define the critical fugacity as the value of the fugacity above which the self-avoiding
walk sticks to the border.
In~\cite{BBDDG} it was proven that
the critical fugacity for the self-avoiding on the hexagonal lattice is equal to~$1+\sqrt{2}$.
It would be natural to generalize this computation.
Though we conjecture that the same should hold, i.\,e. that
the critical fugacity is equal to~$1/(1-2u_1^2)$, we cannot prove this at the moment.

\begin{figure}
\centering
\includegraphics[scale=0.8]{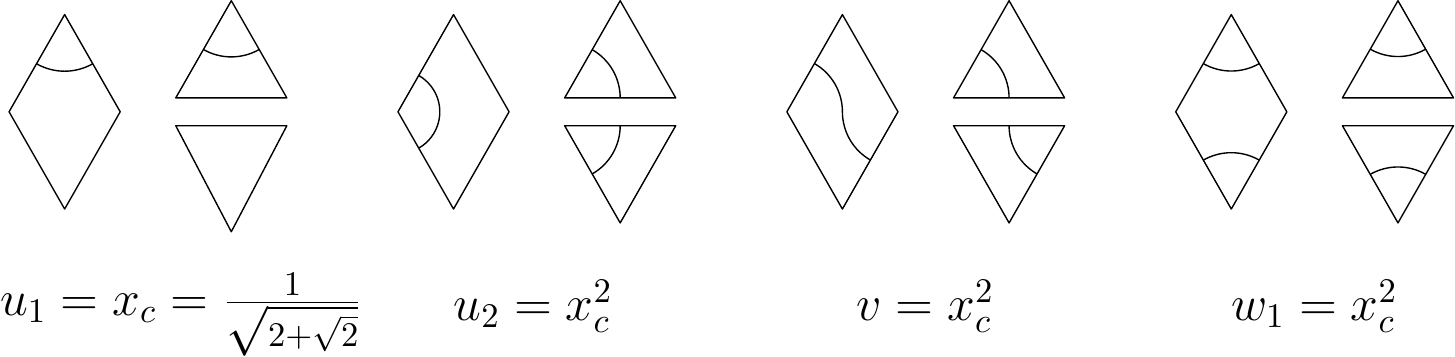}
\begin{subfigure}{.5\textwidth}
  \centering
  \includegraphics[scale=0.85,page=2]{saw_examples.pdf}
\end{subfigure}%
\begin{subfigure}{.5\textwidth}
  \centering
  \includegraphics[scale=0.85,page=3]{saw_examples.pdf}
\end{subfigure}
\caption{\emph{Top}: a bijection between local configurations on rhombi 
    with angle~$\tfrac{\pi}{3}$ and on equilateral triangles.
    Weights are indicated just below the corresponding mapping.
    A rhombus with two~$\tfrac{2\pi}{3}$-arcs is forbidden ($w_2 = 0$).
    \emph{Bottom}: a self-avoiding walk drawn on the triangular and the hexagonal lattices.
    Its length is equal to~17 (compare to fig.~\ref{figWeights}).}
\label{figRhombToHex}
\end{figure}

\section{Case~$\theta=\frac{\pi}{3}$ and the sketch of the proof}
\label{section_honeycomb}
For~$\theta=\frac{\pi}{3}$ we have~$u_1=\frac{1}{\sqrt{2+\sqrt{2}}}$,
$u_2=v=w_1=u_1^2$ and~$w_2=0$. 
This case is in direct correspondence with the self-avoiding walk on the honeycomb lattice,
and Theorem~\ref{th_different_length} specializes to~\cite{DS10}.
We can divide a rhombus with angle~$\frac{\pi}{3}$ into two equilateral triangles. 
Then, all possible states of a rhombus can be viewed as states
of these two triangles (see fig.~\ref{figRhombToHex}).
Note that walking on the faces of the triangular lattice is the same as
walking along the edges of its dual, i.e. of the hexagonal lattice.
Each triangle with an arc of a walk inside it corresponds to a vertex
of the hexagonal lattice visited by a walk.
It is easy to see that the weight of a walk is equal to~$\left(\frac{1}{\sqrt{2+\sqrt{2}}}\right)^{|\gamma|}$.
Therefore, we just obtained the self-avoiding walk on the hexagonal lattice
at criticality.

In order to get a natural length of a walk on the honeycomb lattice,
one needs to fix the length of a~$\tfrac{\pi}{3}$-arc at~1,
and fix the length of a~$\tfrac{2\pi}{3}$-arc and of a straight segment
at~2.\\

Now we turn to the general case~$\theta \in [\frac{\pi}{3},\frac{2\pi}{3}]$.
Let us give an outline of the proof stressing the differences from~\cite{DS10}:

--- In Section~\ref{section_parafermion} we define the parafermionic observable~$F_a$ in
exactly the same way as in~\cite{DS10}.
The only difference is that we consider a walk on a dual graph.

--- In Lemma~\ref{lemWeights} we show that for the weights~\eqref{eq_x1c}-\eqref{eq_w2c}
the parafermionic observable satisfies a part of the discrete Cauchy-Riemann equations~\eqref{CR}.
This means that the contour integral of~$F_a$ around each rhombus is~0.

--- In Lemma~\ref{lem_relation_parallelogramm} (corresponds to Lemma~2 in~\cite{DS10}) 
we sum up this relation over rhombi contained in 
a big parallelogram~$\Omega$ and obtain the relation~\eqref{eqTL} 
on the weights of walks going from the origin to different sides of~$\Omega$ (see fig.~\ref{figParal}).

--- In Lemma~\ref{lem_E_T_is_0} we show that for a 
long parallelogram~$\Omega$ of a fixed width~$T$
the contribution of all walks going to the top and bottom sides is negligible.

--- This leads to the relation~\eqref{eq_relation_strip} on the weights of arcs~$A_T(x_c)$ 
and bridges~$B_T(x_c)$ in a strip, where~$x_c = \tfrac{1}{u_1}$
(corresponds to~(5) in~\cite{DS10}).

--- One can decompose an arc into two bridges and from~\eqref{eq_relation_strip} 
get a lower bound on~$B_T(x_c)$.
This is done in the same way as in~\cite{DS10} but there is a couple of subtleties.
First of all, one is not allowed to do the symmetry around a line of a grid.
In fact, one does not need an axial symmetry, the central symmetry is enough (and this we have).
Another issue is that our walks are allowed to visit the same rhombus twice.
In particular, a walk can visit a rhombus one time before the splitting point and one time after,
and the weight of this walk will not be equal to the product of weights of two bridges.
We need just an upper bound in terms of bridges, so the inequalities~\eqref{ineq1}-\eqref{ineq2}
save the situation.
The last subtlety is that one needs to modify the endpoints of 
the bridges a little bit (see fig.~\ref{figPathsGamma12}).

--- This leads to~$Z(u_1,u_2,v,w_1,w_2) = \infty$, where the parameters are 
given by~\eqref{eq_x1c}-\eqref{eq_w2c} (critical weights).

--- In order to show that~$Z < \infty$ in the subcritical case we do the classical bridges decomposition.
One has to deal with a couple of subtleties that we have already mentioned.
This finishes the proof.

\section{Parafermionic observable and integrable weights}
\label{section_parafermion}
Throughout this section~$\theta \in [\tfrac{\pi}{3},\tfrac{2\pi}{3}]$.

To analyse the behaviour of the self-avoiding walk, 
we will use the~\emph{para\-fermionic} observable introduced in~\cite{S10}. 
Let~$\Omega$ be a parallelogram with angle~$\theta$ divided into
congruent rhombi (see fig.~\ref{figParal}).
Notations: 

--- $V(\Omega)$ is the set of all midpoints of the sides of the rhombi,

--- $V(\partial\Omega)$ is the set of points in~$V(\Omega)$ lying on~$\partial\Omega$
(boundary of~$\Omega$).

Pick points~$a\in V(\partial{\Omega})$ and~$z\in V(\Omega)$ and define:
\begin{align}
\label{eq_parafermion}
F_a(z)=\sum_{\gamma:a\to z}{\omega(\gamma)e^{-i\sigma W(\gamma)}}\, ,
\end{align}
where the sum runs over self-avoiding walks starting at~$a$ and ending at~$z$.
Above, $W(\gamma)$ denotes the winding of~$\gamma$, i.e. the angle of rotation
of~$\gamma$ going from~$a$ to~$z$ 
(a walk crosses all sides of the rhombi at the right angle).
For instance, the arc from~$z_{SE}$ to~$z_{SW}$ on figure~\ref{figCR} has winding~$\theta$
and the arc from~$z_{SE}$ to~$z_{NE}$ has winding~$\theta - \pi$.
The value~$\sigma$ will be fixed later. Observables for other models were introduced in~\cite{CI,CS,S}, see also~\cite{DS11} for a survey.

\begin{figure}[h!]
    \begin{center}
      \includegraphics[scale=0.8]{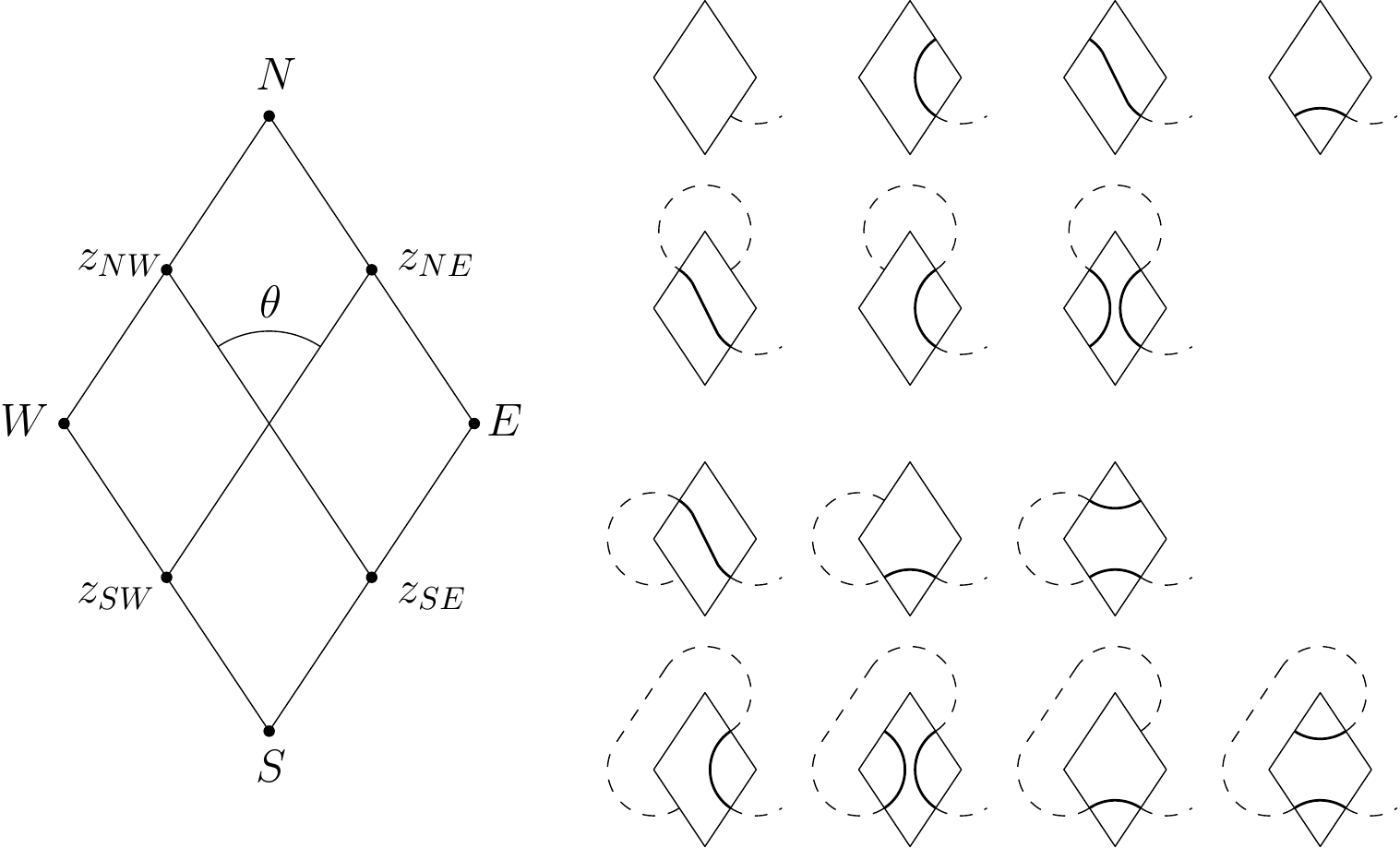}
    \end{center}
    \caption{
    \emph{Left:} A rhombus~$SENW$ with angle~$\theta$ and centres of the edges~$z_{SE}$, $z_{NE}$, $z_{NW}$, $z_{SW}$.
    \emph{Right:} Local changes of the path in different cases~--- each row of rhombi corresponds 
    to one of the equations~\eqref{eqLocal1}-\eqref{eqLocal4}.
    }
    \label{figCR}
\end{figure}%[htb] 

We will find the weights for which our observable satisfies a half of discrete
Cauchy-Riemann equation:
$$
F_a(z_{SE})-F_a(z_{NW})=e^{i\theta}(F_a(z_{SW})-F_a(z_{NE})),
$$
where~$SENW$ is any rhombus with angle~$\theta$ (see fig.~\ref{figCR}). 
We rewrite this equation:
\begin{align}\label{CR}
F_a(z_{SE})+e^{i\theta}F_a(z_{NE})-F_a(z_{NW})-e^{i\theta}F_a(z_{SW})&=0\, .
\end{align}

The other half which is missing is a similar relation around each vertex.

\begin{lemma}
\label{lemWeights}
If~$\sigma=\frac{\ell}{8}$, where~$\ell$ is some odd number, the 
unique weights such that $F_a(u_1,u_2,v,w_1,w_2)$ satisfies~\eqref{CR} 
are given by
\begin{align}
u_1&=\tfrac{1}{t}\sin\left[2\sigma\pi\right]\sin\left[(\sigma-1)(\pi-\theta)\right]\, ,
\label{eq_x1cS}\\
u_2&=\tfrac{1}{t}\sin\left[2\sigma\pi\right]\sin\left[(\sigma-1)\theta\right]\, ,
\label{eq_x2cS}\\
v&=\tfrac{1}{t}\sin\left[(\sigma-1)\theta\right]\sin\left[(\sigma-1)(\pi-\theta)\right]\, ,
\label{eq_vcS}\\
w_1&=\tfrac{1}{t}\sin\left[(\sigma-1)(\pi-\theta)\right]\sin\left[(\sigma-1)(2\pi+\theta)\right]\, ,
\label{eq_w1cS}\\
w_2&=\tfrac{1}{t}\sin\left[(\sigma-1)\theta\right]\sin\left[(\sigma-1)(3\pi-\theta)\right]\, ,
\label{eq_w2cS}
\end{align}
where~$t=\sin\left[(\sigma-1)(\pi+\theta)\right]\sin\left[(\sigma-1)(2\pi-\theta)\right]$.

If~$\sigma = 1$ then there is a one parameter family of weights such that
 $F_a(u_1,u_2,v,w_1,w_2)$ satisfies~\eqref{CR}:
 \begin{align}
 u_1 + u_2 &= 1\, , \label{eq_v0_u1_u2}\\
 w_1 &= u_1\, ,   \label{eq_v0_w1} \\
 w_2 &= u_2\, .  \label{eq_v0_w2}
 \end{align}
 
For all other values of~$\sigma$ the weights such that
$F_a(u_1,u_2,v,w_1,w_2)$ satisfies~\eqref{CR} exist only for some 
specific values of~$\theta$.
\end{lemma}

The weights~\eqref{eq_x1cS}-\eqref{eq_w2cS} give us the 
weights~\eqref{eq_x1c}-\eqref{eq_w2c} if one takes~$\sigma = \tfrac{5}{8}$.

\begin{proof}
Let us consider all paths contributing to Eq.~\eqref{CR} for a fixed rhombus~$SENW$. 
Consider walks visiting~$SENW$ first by~$z_{SE}$ 
(and possibly some other edges of~$SENW$ afterwards). They can be divided into several groups
such that walks in the same group differ only inside~$SENW$ (see fig.~\ref{figCR}):
\begin{itemize}
\item outside~$SENW$ the walk~$\gamma$ is just a path from~$a$ to~$z_{SE}$;
\item outside~$SENW$ the walk~$\gamma$ is  the union of a path from~$a$ to~$z_{SE}$ 
and a path between~$z_{NW}$ and~$z_{NE}$ (visited in any direction);
\item outside~$SENW$ the walk~$\gamma$ is the union of a path from~$a$ to~$z_{SE}$ and 
a path between~$z_{NW}$ and~$z_{SW}$ (visited in any direction);
\item outside~$SENW$ the walk~$\gamma$ is the union of a path from~$a$ to~$z_{SE}$ and a path 
between~$z_{SW}$ and~$z_{NE}$ (visited in any direction).
\end{itemize}
Note that if the total contribution of paths in each group is zero then~$F$ satisfies equation~\eqref{CR}. 
At the same time, in each of these groups, paths differ one from another only inside 
the rhombus~$SENW$. Hence, if the following equations hold, we obtain~\eqref{CR}:
\begin{align}
1+{\lambda}\bar{\mu} e^{i\theta}u_2-v-{\lambda}e^{i\theta}u_1&=0\, , \label{eqLocal1}\\
\lambda\bar{\mu}^2 e^{i\theta}v-\mu u_2-\lambda e^{i\theta}w_2&=0\, , \label{eqLocal2}\\
-\lambda\mu e^{i\theta}v-\bar{\mu} u_1+\lambda\bar{\mu} e^{i\theta}w_1&=0\, , \label{eqLocal3}\\
-\lambda\mu e^{i\theta}u_2-{\mu^2}w_2+\lambda\bar{\mu}^2 e^{i\theta}u_1-\bar{\mu}^2w_1&=0\, , \label{eqLocal4}
\end{align}
where~$\lambda=e^{-i\sigma\theta}$, $\mu=e^{-i\sigma\pi}$. 

Moreover, it is not difficult to see that if~\eqref{eqLocal1}-\eqref{eqLocal4}
are not satisfied then~\eqref{CR} fails for some rhombi.

Now, if we consider walks visiting~$SENW$ first by~$z_{SW}$, we get the
equations that are conjugates to~\eqref{eqLocal1}-\eqref{eqLocal4}
(i.e. the equations with all terms conjugated except the weights~$u_1$, $u_2$, $v$, $w_1$, $w_2$).
For walks visiting~$SENW$ first by~$z_{NW}$ (or~$z_{NE}$) one gets
the same equations as for walks visiting~$SENW$ first by~$z_{SE}$ (or~$z_{SW}$).

Solving this linear system, we obtain that either~$v=0$ or~$\sigma=\frac{\ell}{8}$ 
where~$\ell$ is some odd number. For each~$\sigma = \frac{\ell}{8}$ and~$\theta$ parameters, 
weights satisfying equations~\eqref{eqLocal1}-\eqref{eqLocal4} 
are given by~\eqref{eq_x1cS}-\eqref{eq_w2cS}.
Details are given in the Appendix.

If~$v=0$, the solution exists for each~$\theta$ if and only if~$\sigma = 1$. 
In this case~$w_1+w_2  = 1$, $u_1 = w_1$ and~$u_2 = w_2$.
\end{proof}

\section{Proofs of Theorems~\ref{th_CC}-\ref{th_different_length}}
\label{section_proof}
In order to have positive coefficients in~\eqref{eqTL} and 
to have the inequalities~\eqref{ineq1}-\eqref{ineq2}, we fix~$\sigma$ at the value~$\frac{5}{8}$
till the end of the paper.

In this case, weights given by~\eqref{eq_x1cS}-\eqref{eq_w2cS} 
can be rewritten as the weights given by~\eqref{eq_x1c}-\eqref{eq_w2c}. 
For~$\theta\in[\frac{\pi}{3},\frac{2\pi}{3}]$, all of them are non-negative and
\begin{align}
u_1^2&\ge w_1\label{ineq1}\, ,\\
u_2^2&\ge w_2\label{ineq2}\, .
\end{align}

For~$\theta\in (0, \frac{\pi}{3}) \cup (\frac{2\pi}{3}, \pi)$,
one of~$w_1$ and~$w_2$ is negative and 
one of the inequalities~\eqref{ineq1}-\eqref{ineq2} fails.

Take any~$\theta \in [\tfrac{\pi}{3},\tfrac{2\pi}{3}]$.

We call a self-avoiding walk a {\it bridge} (fig.~\ref{figParal}, walks from~$a$ to~$\beta$)
if it is contained in a strip
such that both endpoints of the walk are contained in different borders
of the strip and these borders go along the lines
of the grid (i.\,e. contain sides of rhombi).

Consider~$x > 0$.
Let us take~$x_c=u_1$, $u_1(x)=x$, $u_2(x)=x\cdot\frac{u_2}{x_c}$, 
$v(x)=x\cdot\frac{v}{x_c}$, $w_1(x)=x^2\cdot\frac{w_1}{(x_c)^2}$ 
and~$w_2(x)=x^2\cdot\frac{w_2}{(x_c)^2}$. 
Denote by~$\omega_x(\gamma)$ the weight of~$\gamma$ if the weights of the plaquets 
are~$x$, $u_2(x)$, $v(x)$, $w_1(x)$ and~$w_2(x)$.
One can observe that for any~$x$ and any self-avoiding walk~$\gamma$ 
of length~$n$ holds~$\omega_x(\gamma)=\left(\frac{x}{x_c}\right)^n\omega_c(\gamma)$, 
where~$\omega_c(\gamma)$ is the weight of~$\gamma$ for~$x=x_c$. 
In order to prove Theorem~\ref{th_CC}, we need to show that the radius of convergence 
of $Z(x)=Z(x,u_2(x),v(x),w_1(x),w_2(x))$ is~$x_c$. 

Now, let us consider a parallelogram~$\Omega$ with angles~$\theta$ and~$\pi-\theta$, 
and sides denoted by~$\alpha$, $\beta$, $\delta$, $\varepsilon$ (see fig.~\ref{figParal}). 
Let~$2L+1$ be the number of rhombi touching~$\alpha$ and~$T$ be the number of 
rhombi touching~$\delta$. The origin~$a$ will be in the middle of~$\alpha$. 
We will use the following notations:
\begin{align*}
A_{T,L}(x)&=\sum_{\gamma:a\to z\in \alpha}{\omega_x(\gamma)},
&B_{T,L}(x)&=\sum_{\gamma:a\to z\in \beta}{\omega_x(\gamma)},\\
D_{T,L}(x)&=\sum_{\gamma:a\to z\in \delta}{\omega_x(\gamma)},
&E_{T,L}(x)&=\sum_{\gamma:a\to z\in \varepsilon}{\omega_x(\gamma)}.
\end{align*}

\begin{figure}[ht]
    \begin{center}
      \includegraphics[scale=0.4]{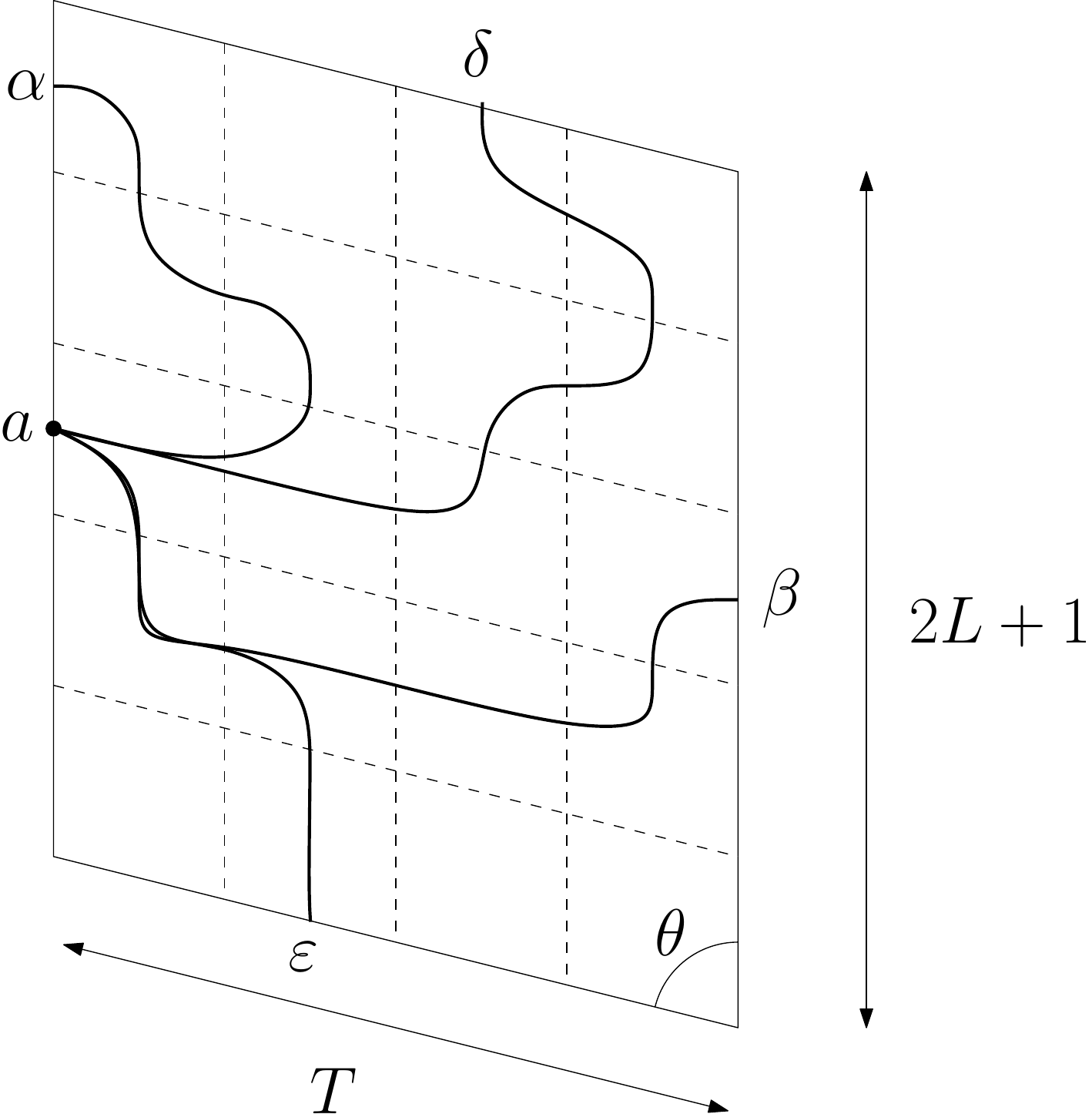}
    \end{center}
    \caption{Parallelogram~$\Omega$ and self-avoiding walks to 
    its sides~$\alpha$, $\beta$, $\delta$ and~$\varepsilon$.
    }
    \label{figParal}
\end{figure}

\begin{lemma}
\label{lem_relation_parallelogramm}
For~$c_\alpha=\cos\frac{3\pi}{8}$, $c_\delta=\cos(\frac{3}{8}\theta)$ 
and~$c_\varepsilon=\cos(\frac{3}{8}(\pi-\theta))$
\begin{align}
c_\alpha A_{T,L}(x_c) + B_{T,L}(x_c) + c_\delta D_{T,L}(x_c) + c_\varepsilon E_{T,L}(x_c)=1\, .\label{eqTL}
\end{align}
\end{lemma}

\begin{proof}
The relation~\eqref{CR} means that the contour integral of~$F_a$
over any rhombus is~0 (by the contour integral we mean the sum
of the values of~$F_a$ along a counter-clockwise oriented contour multiplied by the direction of 
the corresponding edges).
Thus, the contour integral of~$F_a$ over the whole~$\Omega$ is~0.
For the points on the boundary we know the winding.
Taking the imaginary part, we get the desired relation on~$ A_{T,L}(x_c)$, $B_{T,L}(x_c)$,
$ D_{T,L}(x_c)$ and~$ E_{T,L}(x_c)$.
The~1 on the right side is the contribution of the empty walk.
\end{proof}

Note that all three coefficients~$c_\alpha$, $c_\delta$ and~$c_\varepsilon$ are positive.

\begin{remark}
In fact, one can obtain a similar equation for any~$k\in \mathbb{Z}$ 
and~$\sigma = \frac{2k+1}{8}$,
but the coefficients are not always positive.
Also, it is interesting that we are using only the imaginary part of the relation.
One can as well try to derive some information from its real part.
A small difficulty is that now walks to the boundary~$\alpha$ to the left
and to the right of the origin will get different coefficients.
\end{remark}

\begin{lemma}\label{lem_E_T_is_0}
For~$T$ fixed, $E_{T,L}(x_c) \to 0$ and~$D_{T,L}(x_c) \to 0$ as~$L\to \infty$.
\end{lemma}

\begin{proof}
Consider any walk~$\gamma$ contributing to~$E_{T,L}(x_c)$ or~$D_{T,L}(x_c)$.
Let~$\tilde{\gamma}$ be the walk ending at~$\alpha$ obtained from~$\gamma$ by adding
at the end one arc and several (at most~$T-1$) straight segments going leftwards.
It is easy to see that~$\omega_c(\tilde{\gamma}) \ge c_T \omega_c(\gamma)$,
where~$c_T =  v^{T-1}\mathrm{min} (u_1, u_2)$.

Note that~$\tilde{\gamma}$ contributes to~$A_{T,L+1}(x_c) - A_{T,L}(x_c)$
and~$\tilde{\gamma}$ determines~$\gamma$. Thus
\begin{align}
\label{eq_AED}
A_{T,L+1}(x_c) - A_{T,L}(x_c) \ge c_T  (E_{T,L}(x_c) + D_{T,L}(x_c))\, . 
\end{align}
Obviously, the left-hand side is positive and by~\eqref{eqTL} $A_{T,L}(x_c)$
is bounded by~1.
Thus, the left-hand side of~\eqref{eq_AED} tends to~0 as~$L$ tends to~$\infty$ (when~$T$ is fixed).
\end{proof}

For~$x\le x_c$ consider~$A_{T}(x)$ and~$B_T(x)$:
\begin{align*}
A_T(x)&=\lim_{L\to \infty}{A_{T,L}(x)}\, ,
&B_T(x)&=\lim_{L\to \infty}{B_{T,L}(x)}\, .
\end{align*}

These limits exist because~$A_{T,L}(x)$ and~$B_{T,L}(x)$ 
are increasing in~$L$ and bounded by~$1$ (see~\eqref{eqTL}).

We thus obtain
\begin{align}
\label{eq_relation_strip}
c_\alpha A_{T}(x_c) + B_{T}(x_c)&=1\, .
\end{align}

The walks counted in~$B_T$ are self-avoiding bridges of width~$T$.

\begin{lemma}\label{LemAtC}
The partition function is infinite at criticality: $Z(x_c)=\infty$.
\end{lemma}

\begin{proof}
Note that $B_T(x_c)-B_{T+1}(x_c)=c_\alpha (A_{T+1}(x_c) - A_{T}(x_c))$. 
It is easy to see that~$A_{T+1}(x_c) - A_{T}(x_c)$ is the sum of weights of all the self-avoiding paths 
in the strip of width~$T+1$ beginning at~$a$ and ending on the left side of the strip, 
which also touch the right side. 
Each of these paths~$\gamma$ 
can be divided into a path~$\gamma_1$ from the left side of the strip 
to the right one and path~$\gamma_2$ from the right side 
of the strip to the left one  (see fig.~\ref{figPathsGamma12}). 
More precisely, path~$\gamma_1$ is
defined as the part of~$\gamma$ from~$a$ up to (and including) the first visit to the rhombi
on the right boundary of the strip,
and~$\gamma_2 = \gamma \setminus \gamma_1$.

At this point, one must be aware that the weight of~$\gamma$ is not the product of 
weights of~$\gamma_1$ and~$\gamma_2$ since rhombi containing two arcs 
of~$\gamma$ may contain one arc 
of~$\gamma_1$ and one arc of~$\gamma_2$. These rhombi contribute~$w_1$ 
(or~$w_2$) to~$\omega_c(\gamma)$ and~$u_1^2$ (or~$u_2^2$ resp.) 
to~$\omega_c(\gamma_1)\omega_c(\gamma_2)$. 
Nevertheless, the inequalities~\eqref{ineq1} and~\eqref{ineq2} 
imply that~$\omega_c(\gamma)\leq\omega_c(\gamma_1)\omega_c(\gamma_2)$.

Consider the last step of~$\gamma_1$. It is the only time when~$\gamma_1$ visits the right boundary
of the strip.
This implies that the last step in~$\gamma_1$ is either a $\theta$-arc
or a $(\pi-\theta)$-arc (see fig.~\ref{figPathsGamma12}):
\begin{itemize}
\item If the last step in~$\gamma_1$ is a $(\pi-\theta)$-arc, then we define~$\gamma_1'$
as the path obtained by  adding a $(\pi-\theta)$-arc at the end of~$\gamma_1$ 
and~$\gamma_2'$ as the path obtained by adding a $\theta$-arc 
in the beginning of~$\gamma_2$.
\item If the last step in~$\gamma_1$ is a $\theta$-arc, then we define~$\gamma_1'$
as the path obtained by adding a $\theta$-arc at the end of~$\gamma_1$ 
and~$\gamma_2'$ as the path obtained by adding a $(\pi-\theta)$-arc in the beginning 
of~$\gamma_2$.
\end{itemize}

\begin{figure}[ht]
    \begin{center}
      \includegraphics[scale=0.6]{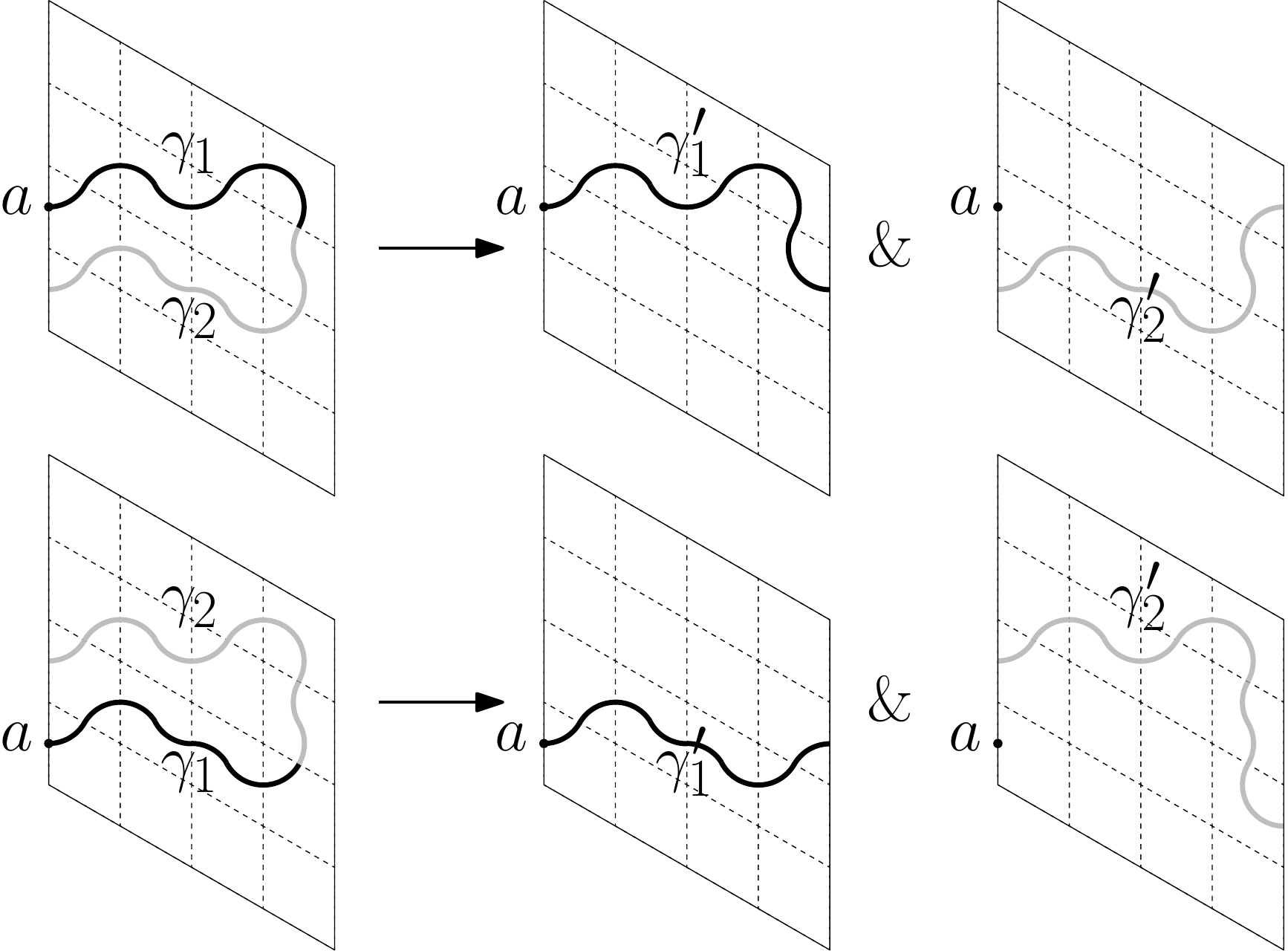}
    \end{center}
    \caption{Two different cases of splitting a walk into~$\gamma_1$ and~$\gamma_2$,
    paths~$\gamma_1'$ and~$\gamma_2'$ in each of the cases.
    }
    \label{figPathsGamma12}
\end{figure}

Paths~$\gamma_1'$ and~$\gamma_2'$ are self-avoiding bridges
of width~$T+1$ starting at some particular points~--- 
$\gamma_1'$ starts at~$a$ and~$\gamma_2'$ starts at the rhombus adjacent 
to the endpoint of~$\gamma_1'$. This leads to the inequality
\begin{align*}
A_{T+1}(x_c) - A_{T}(x_c)&\le (B_{T+1}(x_c))^2/(x_cu_2)\, .
\end{align*}

Using~\eqref{eq_relation_strip}, we obtain a lower bound on the growth of~$B_T(x_c)$:
\begin{align*}
B_T(x_c)-B_{T+1}(x_c)&\le \frac{c_\alpha}{x_cu_2}\cdot(B_{T+1}(x_c))^2\\
\frac{c_\alpha}{x_cu_2}\cdot(B_{T+1}(x_c))^2+B_{T+1}(x_c)&\ge B_T(x_c) \,.
\end{align*}

The last inequality leads to the following bound on~$B_{T+1}(x_c)$ in terms of~$B_T$:
\begin{align}
\label{ineq_BTBT}
B_{T+1}(x_c) \ge \frac{1}{2} (- c + \sqrt{c^2 + 4c B_T(x_c)})
& = \frac{B_T(x_c)}{\frac{1}{2} + \sqrt{\frac{1}{4} + \frac{B_T(x_c)}{c}}}\, ,
\end{align}
where~$c = \frac{x_cu_2}{c_\alpha}$. 
This gives the following bound on~$B_{T+1}(x_c)$:
\begin{align}
\label{ineq_BT}
B_T(x_c)&\ge \frac{1}{T}\min(B_1(x_c),c)\, .
\end{align}
The proof is done by induction, one just needs to use that 
the righthand side in~\eqref{ineq_BTBT} is increasing in~$B_T$ and
to check that the denominator there is not greater than~$\frac{T+1}{T}$
when~$B_T(x_c)$ is replaced by the righthand side of~\eqref{ineq_BT}.

Hence,~$Z(x_c)\ge\sum_{T}{B_T(x_c)}=\infty$ since the harmonic series diverges.
\end{proof}

We give the proof of Theorem~\ref{th_CC} below.

\begin{proof}
It is clear that~$Z(x)=Z(x,u_2(x),v(x),w_1(x),w_2(x))=\sum_{n\ge 0}{\tilde{c}_n x^n}$. 
Because of $\tilde{c}_n\ge \frac{1}{u_1^n} (u_1+v)^n$ (one can use only $\theta$-arcs and straight segments), 
and the submultiplicativity~$\tilde{c}_{n+m}\le \tilde{c}_n\tilde{c}_m$ 
(any path of length~$n+m$ can be divided into a path of length~$n$ and a path of length~$m$: 
the weight of original path is not greater than the product of weights of these 
shorter paths by~\eqref{ineq1}-\eqref{ineq2}), 
there exists~$\tilde{\mu}\in(0,\infty)$ such that~$\tilde{\mu}=\lim \sqrt[n]{\tilde{c}_n}$.

From Lemma~\ref{LemAtC}, we obtain~$Z(x_c)=\infty$. Thus~$\tilde{\mu}\ge x_c^{-1}$.
To get the upper bound on~$\tilde{\mu}$ we need to show that~$Z(x)<\infty$ for~$x<x_c$.

Let us consider~$x<x_c$. Any self-avoiding walk can be decomposed into self-avoiding
bridges, no three of which have the same height (see~\cite{HW}). 
In the original paper it was shown only for usual self-avoiding walks,
but the proof goes through without any changes also in the case
of weighted self-avoiding walks~--- decompose a walk into two half-space walks,
then for each of them pick a bridge of a maximal width, remove them,
note that we are left with two half-space walks of a smaller width,
continue by induction.
One just needs to modify bridges at their endpoints (in exactly the same way
as in Lemma~\ref{LemAtC}), so one gets an additional factor~$(u_1u_2)^{-1}$ 
that we denote by~$c$.

At the same time, the weight of the walk is not greater than the product of weights of these bridges. 
Hence
$$
Z(x)=Z(x,u_2(x),v(x),w_1(x),w_2(x))\le \prod_{T>0}{(1+cB_T(x)})^2\, .
$$
It is clear that~$B_T(x)\le(\frac{x}{x_c})^T\cdot B_T(x_c)\le(\frac{x}{x_c})^T$. 
Thus~$Z(x)\le\prod_{T>0}(1+(\frac{x}{x_c})^T)<\infty$. 
Hence,~$Z(x)<\infty$ for~$x<x_c$ and the proof is finished.
\end{proof}

Note that we never use our particular choice of the definition~$|\gamma|$,
so Theorem~\ref{th_different_length} can be proven along the same lines.

\section{Critical weights for the loop~$O(n)$ model.}
\label{section_O(n)}
We consider a loop representation of the loop~$O(n)$ model on any 
finite simply connected rhombic tiling. 
The configuration in each rhombus is one of those mentioned in fig.~\ref{figWeights} 
and we consider only the configurations which can be decomposed into several loops. 
In this case the weight of the configuration is:
$$
\omega(conf)=\prod_{\text{rhombus } r}{\omega(r)\cdot n^{\#\loops}}\, ,
$$
where~$\omega(r)$ is the weight of rhombus~$r$, i.\,e. either~1 or 
one of~$u_1$, $u_2$, $v$, $w_1$, $w_2$. 
We can also add boundary conditions~--- 
allow paths going from one particular edge on the boundary to another.

Now let us take any~$s$ and~$n=-2\cos{\tfrac{4\pi}{3}s}$.
We consider the following family of weights parametrized by~$s$ and angle~$\theta$ of the rhombus:
\begin{align}
u_1&=\tfrac{1}{t}\cdot\sin{(\pi-\theta)s}\cdot\sin{\tfrac{2\pi}{3}s}\, ,
\label{eq_u_1_n}\\
u_2&=\tfrac{1}{t}\cdot\sin{\theta s}\cdot\sin{\tfrac{2\pi}{3}s}\, ,
\label{eq_u_2_n}\\
v&=\tfrac{1}{t}\cdot\sin{\theta s}\cdot\sin{(\pi-\theta)s}\, ,
\label{eq_v_n}\\
w_1&=\tfrac{1}{t}\cdot\sin{(\tfrac{2\pi}{3}-\theta)s}\cdot\sin{(\pi-\theta)s}\, ,
\label{eq_w_1_n}\\
w_2&=\tfrac{1}{t}\cdot\sin{(\theta-\tfrac{\pi}{3})s}\cdot\sin{\theta s}\, ,
\label{eq_w_2_n}
\end{align}
where
$$
t=\frac{\sin^3{\tfrac{2\pi}{3}s}}{\sin{\tfrac{\pi}{3}s}}+
\sin{(\theta-\tfrac{\pi}{3})s}\cdot \sin{(\tfrac{2\pi}{3}-\theta)s}.
$$

The weights given by~\eqref{eq_u_1_n}-\eqref{eq_w_2_n} coincide with 
the weights given by~\eqref{eq_x1cS}-\eqref{eq_w2cS} for any~$\sigma=\frac{6k+5}{8}$,
where~$k\in\mathbb{Z}$, if one takes~$s=\sigma - 1$.
In particular,~$s=-\tfrac{3}{8}$ gives~\eqref{eq_x1c}-\eqref{eq_w2c}.

One can define~\cite{CI} the parafermionic observable for any~$n$ exactly in 
the same way as we did above for~$n=0$:
$$
F_a(z)=\sum_{\gamma:a\to z}{\omega(\gamma)e^{-i\sigma W(\gamma)}},
$$
where the sum runs over the configurations 
containing only loops and a path from~$a$ to~$z$,
$\omega(\gamma)$ stands for the weight of~$\gamma$ 
and~$W(\gamma)$ denotes the winding of a path in~$\gamma$ going from~$a$ to~$z$.

Another important tool is the Yang-Baxter equation (see~\cite{N90}).
Consider a symmetric equilateral hexagon~$H$. Note that there are two different 
ways to tile it by~3 rhombi.
Denote these two tilings by~$T_1$ and~$T_2$.
We say that the model satisfies the Yang-Baxter equation if for any fixed configuration
outside of~$H$ the sum of the weights of all its possible extensions to~$H$ 
is the same for tilings~$T_1$ and~$T_2$.

\begin{proposition}
Let~$s\in\mathbb{R}$ and take~$n=-2\cos{\tfrac{4\pi}{3}s}$.
Then the loop~$O(n)$ model with the weights given by~\eqref{eq_u_1_n}-\eqref{eq_w_2_n}
satisfies Yang-Baxter equation and the parafermionic observable~$F$ 
with spin~$\sigma = s+1$ satisfies the following equation on any rhombus~$SENW$
(see fig.~\ref{figCR}):
$$
F_a(z_{SE})+e^{i\theta}F_a(z_{NE})-F_a(z_{NW})-e^{i\theta}F_a(z_{SW})=0.
$$
\end{proposition}

The proof for the parafermionic observable can be done 
by local transformations in the same way as the proof 
of Lemma~\ref{lemWeights}. 
One should just keep in mind that there are more different local configurations in this case
because of loops.

For the connection between the parafermionic observable and the Yang-Baxter
relation see~\cite{AB}.

\begin{remark}
The weights are symmetric in~$\theta$~--- if one takes~$\pi-\theta$ instead of~$\theta$ 
then~$v$ is the same, $u_1$ and~$u_2$ are exchanged and~$w_1$ and~$w_2$ are exchanged.

One can see that for~$\theta=\tfrac{\pi}{3}$ and any~$s$ the weights can be factorized, 
i.\,e. $w_1=v=u_2=u_1^2$, $w_2=0$ and~$u_1=\pm\tfrac{1}{\sqrt{2\pm\sqrt{2-n}}}$. 
So this is just the loop~$O(n)$ model on the honeycomb lattice with the weight for each edge 
being equal to~$\pm\tfrac{1}{\sqrt{2\pm\sqrt{2-n}}}$ (see fig.~\ref{figRhombToHex}).
Nienhuis nonrigorously derived~\cite{N82} $\tfrac{1}{\sqrt{2+\sqrt{2-n}}}$ to be the critical 
value for the loop~$O(n)$ model on the honeycomb lattice.
\end{remark}

\renewcommand{\thesection}{A}
\setcounter{equation}{0}
\section{Appendix}

\subsection*{Computations in Lemma~\ref{lemWeights}}

Note that Eq.~\eqref{eqLocal1}, \eqref{eqLocal4} and the conjugates of 
Eq.~\eqref{eqLocal2}-\eqref{eqLocal3} can be rewritten in the following way:
\begin{align}
1+{\lambda}\bar{\mu} e^{i\theta}u_2-v-{\lambda}e^{i\theta}u_1&=0\, ,
\label{eq_v_u1_u2}\\
\bar{\mu}^2 (\lambda \bar\mu  e^{i\theta} u_2 + w_2) &= v\, ,
\label{eq_v_u2_w2}\\
\mu^2(-\lambda  e^{i\theta} u_1+ w_1 ) &= v\, ,
\label{eq_v_u1_w1}\\
\mu^2(\lambda\bar\mu e^{i\theta}u_2 + w_2)+\bar{\mu}^2( - \lambda e^{i\theta}u_1 + w_1)&=0\, .
 \label{eq_u1_u2_w1_w2}
\end{align}

It is easy to see that~\eqref{eq_v_u2_w2}-\eqref{eq_u1_u2_w1_w2}
are equivalent to~\eqref{eq_v_u2_w2}-\eqref{eq_v_u1_w1} plus the following relation:
\begin{align}
v(\mu^4 + \bar{\mu}^4) = 0.\label{eq_v_sigma}
\end{align}

First, consider the case~$v\ne 0$. Then~\eqref{eq_v_sigma} gives us the desired condition
on~$\sigma$:
$$
\cos(4\sigma\pi) = 0.
$$

Equations~\eqref{eq_v_u2_w2}-\eqref{eq_v_u1_w1} and their conjugates allow us to express
everything in terms of~$v$ and some trigonometric functions:
\begin{align*}
u_1 &= v\cdot \frac{\sin( - 2\sigma\pi)}{\sin((1 - \sigma)\theta)}\, , \\
u_2 &= v\cdot \frac{\sin( - 2\sigma\pi)}{\sin( \sigma\pi + (1 - \sigma) \theta )}\, ,\\
w_1 &= v\cdot \frac{\sin((1- \sigma )\theta - 2\sigma\pi)}{\sin((1 - \sigma)\theta)}\, ,\\
w_2 &= v\cdot \frac{\sin((1 - \sigma )\theta + 3\sigma\pi)}{\sin( \sigma\pi + (1- \sigma)\theta  )}\, .
\end{align*}

Then~\eqref{eq_v_u1_u2} gives us a linear equation on~$v$.
It is straightforward to check that the solution is unique and given by~\eqref{eq_x1cS}-\eqref{eq_w2cS}.

If~$v=0$, equations are transformed into
\begin{align*}
u_1 &=\lambda e^{i \theta} w_1\, , \\
u_2 &=-\lambda\bar{\mu}e^{i \theta} w_2\, , \\
w_1 + w_2 &= 1\, .
\end{align*}

We know that the conjugated equations should also hold.
Thus, in order to have a solution for each~$\theta$ we have to set~$\sigma = 1$.
In this case~$u_1=w_1$ and~$u_2 = w_2$.

%%%%%%%%%%%%%%%%%%%%%%%%%%%%%%%%%%%%%%%%%%%%%%%%%%%%%%%%%%%%%%%%%%%
%%                                                               %%
%% Use the two commands below for producing your bibliography    %%
%% with bibtex, then comment again the commands and include the  %%
%% content of the .bbl file in this file below the commands.     %%
%%                                                               %%
%%%%%%%%%%%%%%%%%%%%%%%%%%%%%%%%%%%%%%%%%%%%%%%%%%%%%%%%%%%%%%%%%%%

%\bibliographystyle{amsplain}
%\bibliography{yourbibfilename}

% add below the content of your .bbl file produced by bibtex.

%%%%%%%%%%%%%%%%%%%%%%%%%%%%%%%%%%%%%%%%%%%%%%%%%%%%%%%%%%%%%%%%%%%
%%                                                               %%
%% You may add acknowledgments (optional).                       %%
%%                                                               %%
%%%%%%%%%%%%%%%%%%%%%%%%%%%%%%%%%%%%%%%%%%%%%%%%%%%%%%%%%%%%%%%%%%%

\ACKNO{
I am grateful to Stanislav Smirnov for introducing me to the subject and sharing the ideas, 
to Dmitry Chelkak for many valuable and encouraging discussions, 
to Hugo Duminil-Copin for fruitful discussions and comments on the draft version of this paper. 
This research was supported by the NCCR SwissMAP, the ERC AG COMPASP, the Swiss NSF
and Chebyshev Laboratory at Saint Petersburg State University under the Russian Federation 
Government grant 11.G34.31.0026.}

%%%%%%%%%%%%%%%%%%%%%%%%%%%%%%%%%%%%%%%%%%%%%%%%%%%%%%%%%%%%%%%%%%%
%%                                                               %%
%% You have reached the end of your document.                    %%
%%                                                               %%
%%%%%%%%%%%%%%%%%%%%%%%%%%%%%%%%%%%%%%%%%%%%%%%%%%%%%%%%%%%%%%%%%%%

\end{document}